\documentclass[11pt,reqno]{amsart}
\usepackage[utf8]{inputenc}
\usepackage{amsmath,amssymb}
\usepackage{wrapfig}
\usepackage{url}
\usepackage{mathtools}
\usepackage{graphicx}
\usepackage{stmaryrd}
\usepackage{amsthm}
\usepackage{xcolor}
\usepackage[colorlinks=true,linkcolor=blue,citecolor=blue]{hyperref}
\usepackage[shortlabels]{enumitem}
\usepackage{relsize}
\usepackage{dsfont}
\usepackage{stmaryrd}
\usepackage{geometry}
\usepackage{setspace}
 \usepackage{relsize}
 \usepackage{float}
\usepackage{mathrsfs} 
\usepackage{slashed}
\usepackage{scrextend}
\usepackage{mathtools}
\newcommand{\defeq}{\vcentcolon=}

\usepackage{makecell}
\setcellgapes{4pt}
\deffootnote{0em}{1.6em}{\enskip}
\allowdisplaybreaks

\newtheorem{theorem}{Theorem}[section]
\newtheorem{corollary}[theorem]{Corollary}

\theoremstyle{definition}
\newtheorem{obs}{Remark}
\theoremstyle{definition}

\newcommand{\C}{\mathbb{C}}
\newcommand{\R}{\mathbb{R}}

\newcommand{\N}{\mathbb{N}}

\newcommand{\sph}{\mathbb{S}^{2}}
\newcommand{\sphn}{\mathbb{S}^{n}}

\date{\today}

\author{Uwe Kähler}
\address[Uwe Kähler]{
Center for Research and Development in Mathematics and Applications (CIDMA), 
Department of Mathematics, University of Aveiro, 3810-193 Aveiro, Portugal}
  
\email{ukaehler@ua.pt}

\author{André Pedroso Kowacs}
\address[André Pedroso Kowacs]{Departamento de Matemática,
 Instituto de Ciências Matemáticas e de Computação,
 Universidade de São Paulo,
 São Paulo, Brazil}
\email{andrekowacs@gmail.com}

\author{Michael Ruzhansky}
\address[Michael Ruzhansky]{
  Department of Mathematics: Analysis, Logic and Discrete Mathematics,
  Ghent University, Belgium and 
  Queen Mary University of London, 
   United Kingdom}
  
\email{Michael.Ruzhansky@ugent.be}

\thanks{Andr\'e Pedroso Kowacs was supported, in part, by the São Paulo Research Foundation (FAPESP), Brazil,
Process Number \#2025/08151-5. Michael Ruzhansky was supported, in part, by the FWO  Odysseus  1  grant  G.0H94.18N:  Analysis  and  Partial Differential Equations, by the Methusalem programme of the Ghent University Special Research Fund (BOF)
(Grant number 01M01021) and by the FWO grant G011522N. Michael Ruzhansky was also supported by EPSRC grant 
UKRI3645 and by the FWO-FAPESP Bilateral Scientific Cooperation Research Grant G0AOZ25N. Uwe Kaehler is supported by FWO grant V500824N as well as (partially) supported by CIDMA under the Portuguese Foundation for Science and Technology 
(FCT, https://ror.org/00snfqn58)   Multi-Annual Financing Program for R\&D Units,
grants UID/4106/2025 and UID/PRR/4106/2025.}



\title[Lieb-Thirring inequalities for the Dirac operator]{Lieb-Thirring inequalities for the Dirac operator on spheres}

\begin{document}

 \begin{abstract}
 In this paper, we obtain bounds for the best  constants in two inequalities which can be seen as  analogues of the Lieb-Thirring inequality, but with the Dirac operator, on the $n-$sphere. We then apply these results in order to improve the known upper bounds on the classical Lieb-Thirring constant on the $n$-sphere for $n\geq 5$.
 \end{abstract}

 \maketitle
 \section{Introduction}

Lieb-Thirring inequalities, first introduced in \cite{Lieb2002}, offer bounds on the $\gamma$-moments of the negative eigenvalues for the Schrödinger operator:
\begin{equation}\label{deltav}
\Psi = -\Delta - V,
\end{equation}
acting in the Hilbert space $L^2(\R^n)$. These inequalities assert the existence of constants $L_{\gamma,n} > 0$ such that
\begin{equation}\label{Lieb1}
\sum_{\lambda_j \leq 0} |\lambda_j|^\gamma \leq L_{\gamma,n} \int_{\mathbb{R}^n} V(x)^{\gamma + \frac{n}{2}} dx,
\end{equation}
where $V \geq 0$ is a potential function that diminishes rapidly at infinity, and $\gamma \geq \max\{1-\frac{n}{2},0\}$. These inequalities have numerous applications in fields such as dynamical systems (\cite{application3, application2, application1}) and quantum mechanics. The exact values for $L_{\gamma,n}$, especially for $\gamma > \frac{3}{2}$ and for any $n \in \N$, were determined in \cite{sharpvalue} and are expressed as:
\begin{equation}\label{sharprn}
    L^{cl}_{\gamma,n} = \frac{\Gamma(\gamma+1)}{2^n\pi^{\frac{n}{2}}\Gamma(\gamma+n/2+1)}.
\end{equation}
This inequality is equivalent to a related inequality for orthonormal families of functions. Specifically, there is a constant $k_n > 0$ such that for any orthonormal set of functions $\{\psi_j\}_{j=1}^N \subset H^1(\R^n)$, and with $\rho(x) \defeq \sum_{j=1}^N |\psi_j(x)|^2$, we have:
\begin{equation}\label{LiebGenereal}
    \int_{\mathbb{R}^n} \rho(x)^{\frac{n+2}{n}}dx \leq k_n \sum_{j=1}^N \|\nabla \psi_j\|^2_2.
\end{equation}
The best constants $k_n$ are connected to $L_{\gamma,n}$ as follows (see \cite{application1, Lieb2002}):
\begin{equation*}
    k_n = \frac{2}{n} \left(1 + \frac{n}{2}\right)^{\frac{n+2}{n}} L^{\frac{2}{n}}_{1,n}.
\end{equation*}
These types of inequalities can also be generalized to manifolds. For a smooth manifold $M$ of dimension $n$, let $k_{M} > 0$ denote the smallest constant for which the following inequality holds for any set of orthonormal functions $\{\psi_j\}_{j=1}^N \subset H^1(M)$, with $\rho(x) \defeq \sum_{j=1}^N |\psi_j(x)|^2$:
\begin{equation}\label{LiebGenereal2}
    \int_{M} \rho(x)^{\frac{n+2}{n}}dx \leq k_{M} \sum_{j=1}^N \|\nabla \psi_j\|_{L^2}^2.
\end{equation}

However, for compact manifolds $M$, the zero eigenvalue must be considered. Thus, instead of the operator \eqref{deltav}, we use:
\begin{equation*}
    -\Delta - \Pi(V \Pi \cdot),
\end{equation*}
where {$\Delta$ now denotes the Laplace-Beltrami operator on $M$ and} $\Pi$ is the orthogonal projection given by:
\begin{equation*}
    \Pi \psi(x) = \psi(x) - \frac{1}{|M|} \int_M \psi(y) dy.
\end{equation*}
For further reading on this topic, see \cite{survey, nsphere, liebsome, hyperbolic, torus, LiebonS2, onaclass}.

In recent years Dirac operators with a potential as the natural relativistic counterpart of Laplacians received a lot of attention from the point of view of the spectral theory. These operators describe relativistic spin-1/2 particles propagating under strongly localized external potentials, such quark confinement. For instance, Dirac operators with singular electrostatic, Lorentz scalar and anomalous magnetic interactions supported on surfaces in $\mathbb{R}^3$ have been studied in~\cite{Behrndt1}, while Dirac operators with $\delta$-shell potential were treated in~\cite{Behrndt2}. These operators also appear in the study of conductivity properties of single-layer graphene structures. The $n$-dimensional Dirac operator with potential as approximations of the Hamiltonians of
interactions of relativistic quantum particles with potentials has been studied in~\cite{Rabinovich1}.

In another direction the study of spectra of Dirac operators over Riemannian manifolds is a classic topic, see, e.g.~\cite{Ginoux} or~\cite{Friedrich}. In general only estimates are known, but in special cases like the $n$-dimensional sphere explicit values for the eigenvalues of the Dirac operator $\slashed D$ are known~\cite{Bär}. 

This leads to the question if we can obtain the corresponding analogue for the Lieb-Thirring inequality for the Dirac operator $\slashed D$. Since the spectrum of the Dirac operator contains both positive and negative eigenvalues such an analogue can only be obtained for the positive or negative part of the spectrum. Therefore, our concrete problem is given by:

Does there exist a constant $K_{\mathbb{S}^n}>0$ such that 
 \begin{equation*}
     \sum_{j=1}^N\langle\Lambda^+(\slashed D)\psi_j,\psi_j\rangle_{L^2} \geq K_{\mathbb{S}^n}\int_{\sphn}\rho^{\frac{n+1}{n}}(x)dx,
\end{equation*}
for any orthonormal family $\{\psi_j\}_{j=1}^N$ of functions in $H^1(\mathbb{S}^n)$, each with mean value zero, and if so, what is the best constant $K_{\mathbb{S}^n}$? Hereby $\Lambda^+(\slashed D)$ denotes the projection onto the positive spectrum $\Lambda^+(\slashed D)=\chi_{(0,+\infty)}(\slashed D)$.  The second question is if under the same conditions, do we have a best constant $K_{\mathbb{S}^n}'>0$ such that
 \begin{equation*}
    \sum_{j=1}^N \|\slashed D\psi_j\|_{L^2}^2=\sum_{j=1}^N\langle \slashed D\psi_j, \slashed D \psi_j\rangle \geq  K_{\sphn}'\int_{\sphn}\rho^{\frac{n+2}{n}}(x)dx
\end{equation*}
and if so what is the best constant?

 The technique employed in the proofs of this paper drew significant inspiration from \cite{torus,LiebonS2}, where the authors established the upper limits $k_{\mathbb{S}^2} \leq \frac{3}{2\pi}$ and $k_{\mathbb{T}^2} \leq \frac{6}{\pi^2}$. Their approach incorporated concepts reminiscent of those initially presented by M. Rumin \cite{Rumin}, who considered the inequality in its original form \eqref{Lieb1}, within a broader context and applied to $\R^n$. This methodology was subsequently refined for the Euclidean framework in \cite{frank} and \cite{Frank2}.

Finally, making use of the fact that the Dirac operator coincides with the gradient operator for real-valued functions, we are able to estimate the classical Lieb-Thirring constant $k_{\sphn}$ from our estimates on $K_{\sphn}'$ and improve the best known upper bounds for $n\geq 5$ presented in \cite{Kow}. In particular, we prove that (see Theorem \ref{teoLiebsn_improved}):
\begin{theorem}
The best constant $k_{\sphn}$ satisfies
   \begin{equation*}
    k_{\sphn}\leq 
           \left(\frac{2n^4+16n^3+42n^2+44n+16}{n^4-6n^2+6n+8}\right)n!^{-\frac{2}{n}}2^{\frac{2}{n}(\lfloor \frac{n}{2}\rfloor+1)-1}(\sigma_{\sphn})^{-\frac{2}{n}},
    \end{equation*}
    for $n\geq 2$, where $\sigma_{\sphn}$ denotes the volume of the $n$-sphere.
\end{theorem}

\section{Preliminaries}

 For $d\in\N$, let $\mathbb{R}_{0,d}$  be the $2^d$-dimensional universal real Clifford algebra  over $\mathbb{R}^d$  constructed from the basis $\{e_1, e_2,\ldots, e_d\}$ under the usual relations$$
e_ke_l+e_le_k=-2\delta_{kl},~ 1\leq k, l\leq d,
$$
where $\delta_{kl}$ is the Kronecker delta function. An element $f\in\mathbb{R}_{0,d}$ can be represented as  $f=\sum_{A}f_Ae_A, f_A\in\mathbb{R},$ where $e_A=e_{j_1j_2\ldots j_k}=e_{j_1}e_{j_2}\ldots e_{j_k}$, $A=\{j_1,j_2,\ldots j_k\}$ with $1\leq j_1\leq j_2 \leq \cdots\leq j_k\leq d$, and $e_0=e_\emptyset=1$ is the  identity element of $\mathbb{R}_{0,d}$. Also consider the complexified Clifford algebra $\C_d=\R_{0,d}\otimes \C$. The elements of the algebra $\mathbb{R}_{0,d}$ for which $|A|=k$ are called $k$-vectors, and similarly for $\C_d$. We denote the space of all $k$-vectors by
\begin{equation*}
\mathbb{R}_{0,d}^k  :=   \text{ span}_{\mathbb{R}}\{e_A:|A|=k\},
\end{equation*}
and analogously $\C_d^k$ for complex $k$-vectors.

 For $\mathbb{K}=\R$ or $\mathbb{K}=\C$, it is clear that the spaces $\mathbb{K}$ and $\mathbb{K}^d$ can be identified with $\mathbb{K}_{0,d}^0$ and  $\mathbb{K}_{0,d}^1,$ respectively. Elements of the space $\mathbb{K}_{0,d}^0\oplus\mathbb{K}_{0,d}^1$ are also called para-vectors (scalar+vector). 
 
 We define the conjugation as an anti-automorphism on $\R_{0,d}$ and $\C_d$ by
 \begin{equation*}
     f\mapsto\overline{f}=\sum_A\overline{f_A}\overline{e_A},
 \end{equation*}
 where the conjugation on the basis elements $e_A$ is defined by the properties: $\overline{e_j}=-e_j,$ for $j=1,\ldots, d$, and  $\overline{ab}=\overline{b}\overline{a}$ for every $a,b\in\mathbb{R}_{0,d}$ or $\C_d$, respectively. 

In what follows, we will require two types of scalar products. First, we introduce the (complex-valued) scalar product of  $a,b\in\mathbb{R}_{0,d}$ or $\C_d$, as the scalar part of their geometric product  
\begin{equation*}
a \cdot b := [\overline{a}{b}]_0=\sum_{A}\overline{a_A} {b_A}. 
\end{equation*} As usual, when we set $a=b$ we obtain the square of the modulus (or magnitude) of the multivector $a$:
\begin{equation*}
|a|^2=[\overline{a}a]_0=\sum_{A}|a_A|^2.
\end{equation*}

Let $\Omega\subset \mathbb{R}^n$ be a (bounded or unbounded) domain. Any function $f:\Omega\to\mathbb{C}_{n}$ can be represented by $f=\sum_A e_Af_A$ with complex-valued component functions $f_A:\Omega\to\mathbb{\C}$. Thus notations such as $f\in C^k(\Omega,\mathbb{C}_{d})$, $L^p(\Omega, \mathbb{C}_{d})$, 
can be understood both co-ordinate-wisely and directly.

In particular, we denote by  $L^p(\Omega, \mathbb{C}_{d})$ the right-linear module of all Clifford-valued functions $f:\Omega\rightarrow\mathbb{C}_{d}$
with finite norm 
\begin{equation*}
\| f \|_{L^p(\C_{d})}=\left\{\begin{array}{ll}
\left( \int_{\Omega}|f(\mathbf{x})|^p d^n\mathbf{x} \right)^{\frac{1}{p}}, & {1\leq p<\infty}, \\
{\rm ess} \sup_{\mathbf{x}\in\Omega}|f(\mathbf{x})|, & {p=\infty}, \\
\end{array}\right.
\end{equation*} where $d^n\mathbf{x}=dx_1dx_2\ldots dx_n$ represents the usual Lebesgue measure in $\mathbb{C}^n$. In the case of $p=2,$ we shall also denote this norm by $\| f\|$. 

Given two functions $f, g \in L^2(\Omega,\mathbb{C}_{d}),$ we define a Clifford-valued  sesqui-linear form via
$$ ( f,g):=\int_{\Omega}{\overline{f(\mathbf{x})}g(\mathbf{x})}d^n\mathbf{x},$$
from which we  construct the scalar inner product
\begin{equation}\label{scalar inner product}
\langle f,g\rangle:\,=\left[(f,g)\right]_0=\int_{\Omega}[\overline{f(\mathbf{x})}g(\mathbf{x})]_0d^n\mathbf{x}.
\end{equation} We remark that (\ref{scalar inner product}) satisfies the (Clifford-) Cauchy-Schwarz inequality
 \begin{equation*}
|\langle f,g\rangle|\leq K_d\|f\|\|g\|,\quad \forall f, g\in L^2(\Omega,\mathbb{C}_{d}),
\end{equation*}
with best constant $K_d$ given by (see \cite{best_constants_complex,best_constants}):
\begin{equation*}
    K_d=\begin{cases}
2^{d/4} & \text{if } d \text{ is even},\\
2^{(d+1)/4} & \text{if } d \text{ is odd.}
\end{cases}
\end{equation*}

Let us remark that in the special case where $f$ or $g$ are either vector- or para-vector-valued functions we have the classical Cauchy-Schwarz inequality
$$
|\langle f,g\rangle|\leq \|f\|\|g\|.
$$
We also need to point out that in general both inner products are needed. While $\langle \cdot,\cdot\rangle$ gives rise to a norm, for questions of duality the inner product $(\cdot,\cdot)$ is needed. In particular, in the case of $\Omega=\sphn$ we have the spaces $L^p(\sphn,\C_{d})$, where $\sphn$ denotes the $n$-dimensional sphere. As a compact manifold without boundary, for $s\in\R$ we also consider the corresponding Sobolev spaces $H^s(\sphn,\C_d)=W^{s,2}(\mathbb{S}^n,\C_d)$, which as before can be understood both co-ordinate-wisely and directly.

    
In this paper we will consider the spherical Dirac operator~\cite{Bär}. For this we will introduce the Gamma operator
$$
\Gamma=-\sum_{j<k}e_{jk}L_{jk}=-\sum_{j<k}e_{jk}(x_j\partial_{x_k}-x_k\partial_{x_j})
$$
where $L_{jk}=x_j\partial_{x_k}-x_k\partial_{x_j}$ are the angular momentum operators. Using this operator the spherical Dirac operator over $\mathbb{S}^n$ has the form
$$
\slashed D=\Gamma-\frac{n-1}{2}.
$$
In particular, we have that its square satisfies
$$
\slashed D^2=(\Gamma-\frac{n-1}{2})^2=(\frac{n-1}{2})^2-\Gamma(n-1-\Gamma)=(\frac{n-1}{2})^2-\Delta_{LB},
$$
where $\Delta_{LB}$ denotes the Laplace-Beltrami operator over the sphere. In the notation of~\cite{Bär} this operator corresponds to the operator $D+\mu$ with $\mu=1/2$ and the above identity is known as the Weitzenböck formula (Lemma 2 in~\cite{Bär}).
\section{Main Results}

 In this section we present the main results obtained in this paper. Their proofs can be found in Section \ref{section_proofs}. 

 For the rest of this paper, let $d\in\N$ be fixed. Also, we denote by $\nabla$ the spherical gradient, that is, $\nabla  \equiv \nabla_{\sphn}$.

\begin{theorem}\label{teos2}
 For any orthonormal family $\{\psi_j\}_{j=1}^N\subset H^{1}(\sph,\C_d)$, such that $\int_{\sph} \psi_j=0$, for $1\leq j\leq N$, we have that
    \begin{equation*}
     \sum_{j=1}^N\langle\Lambda^+(\slashed D)\psi_j,\psi_j\rangle \geq \frac{1}{3}\int_{\sph}\rho(x)^{\frac{3}{2}}dx,
\end{equation*}
where
$\rho(x)\defeq\sum_{j=1}^N|\psi_j(x)|^2,$ and  $\Lambda^+(\slashed D)=\chi_{(0,+\infty)}(\slashed D)$ is the projection onto the positive spectrum of the Dirac operator. 
\end{theorem}

Next, we generalize Theorem \ref{teos2} for the $n$-sphere, where $n\in\N$ is arbitrary. The case $n=2$ was treated separately in order to obtain better estimates in this case of great importance in physics and mathematics.

\begin{theorem}\label{teosn}
For any orthonormal family $\{\psi_j\}_{j=1}^N \subset H^{1}(\sphn,\C_d)$ and $\int_{\sphn}\psi_j=0$ for $1\leq j\leq N$, we have that
    \begin{equation*}
     \sum_{j=1}^N\langle\Lambda^+(\slashed D)\psi_j,\psi_j\rangle \geq   {{n!}^{\frac{1}{n}}2^{-\frac{1}{n}\lfloor \frac{n}{2}\rfloor}}c_n\int_{\sphn}\rho^{\frac{n+1}{n}}(x)dx,
\end{equation*}
where 
$\rho(x)\defeq\sum_{j=1}^N|\psi_j(x)|^2,$  $\Lambda^+(\slashed D)=\chi_{(0,\infty)}(\slashed D)$  is the projection onto the positive spectrum of the Dirac operator and where $c_n$ is given by:
     \begin{equation}
     \begin{aligned}\label{eq_mn}
         c_n&=\frac{n^2}{n^2+3n+2}-\frac{(\sqrt{n+1}-\sqrt{n})^{\frac{2}{n}}}{{(n+1)^{\frac{n+1}{n}}}}\frac{n(2\sqrt{n(n+1)}-n)}{n^2+3n+2},
         \end{aligned}
     \end{equation}
      and the integration is taken with respect to the normalized standard surface measure on the sphere.
\end{theorem}

\begin{theorem}\label{teosn2}
For any orthonormal family $\{\psi_j\}_{j=1}^N \subset  H^{1}(\sphn,\C_d)$ and $\int_{\sphn}\psi_j=0$ for  $1\leq j\leq N$, we have that
    \begin{equation*}
    \sum_{j=1}^N \|\slashed D\psi_j\|_{L^2(\C_d)}^2=\sum_{j=1}^N\langle \slashed D\psi_j, \slashed D \psi_j\rangle \geq  c_n'n!^{\frac{2}{n}}2^{{1}-\frac{2}{n}(\lfloor\frac{n}{2}\rfloor+1)}\int_{\sphn}\rho^{\frac{n+2}{n}}(x)dx,
\end{equation*}
where
$\rho(x)\defeq\sum_{j=1}^N|\psi_j(x)|^2$, 
\begin{equation}\label{eq_cn'}
    c_n'=\begin{cases}
           \frac{n^4-6n^2+6n+8}{2n^4+16n^3+42n^2+44n+16},&\text{ if }n\geq 2,\\
           \frac{7}{360}&\text{ if } n=1,
    \end{cases}
\end{equation}
and the integration is taken with respect to the normalized standard surface measure on the sphere.
\end{theorem}

\begin{obs}
    We note that Theorems \ref{teos2}  and \ref{teosn} remain valid if we replace $\C_d$ by $\R_{0,d}$.
\end{obs}

Using Theorems \ref{teos2}, \ref{teosn} and \ref{teosn2}, we obtain the following approximation for the lower bounds for $K_{\sphn}$ and $K'_{\sphn}$ for the first few values of $n$.
\makegapedcells
\begin{table}[H]
\begin{tabular}{||c| c||} 
 \hline
 $n$&$K_{\sphn}\geq$ \\ [0.5ex] 
 \hline\hline
1&0.153595\\
\hline
2&0.$\overline{333333}$\\
\hline
3&0.593385\\
\hline
4&0.767663\\
\hline
5&1.087326\\
 \hline
\end{tabular}\quad
\begin{tabular}{||c| c||} 
 \hline
 $n$&$K_{\sphn}'\geq$ \\ [0.5ex] 
 \hline\hline
1&0.0097$\overline{222}$\\
\hline
2&0.02$\overline{77777}$\\
\hline
3&0.1240172\\
\hline
4&0.2771281\\
\hline
5&0.6682066\\
 \hline
\end{tabular}

 \caption{Lower bounds for the best constants $K_{\sphn}$ and $K_{\sphn}'$.}\label{table1}
 \end{table}
Note that the lower bound given by Theorem \ref{teos2} is sharper than the one obtained by Theorem \ref{teosn} for the case $n=2$, due to the less precise estimates for $\|\Gamma_\psi P_E\chi_B\|_{HS}$ in the proof of the latter theorem.

Also, as $n\to \infty$, the lower bounds obtained by Theorem \ref{teosn} and Theorem \ref{teosn2} have growth order of $n!^{\frac{1}{n}}$ and $n!^{\frac{2}{n}}$, as $n\to\infty$, respectively.

As a special case we have the following corollary, where $\|\cdot\|_{L^2}$ denotes the usual $L^2$ norm for scalar-valued functions on the sphere, with respect to the normalized standard surface measure. 

\begin{corollary}\label{coro_scalar}
For any orthonormal family  of complex-valued functions $\{\psi_j\}_{j=1}^N \subset  H^{1}(\sphn)$ and $\int_{\sphn}\psi_j=0$ for all $j$, we have that
    \begin{equation*}
     \sum_{j=1}^N\|\nabla\psi_j\|^2_{L^2}=\sum_{j=1}^N\|\slashed D\psi_j\|_{L^2}^2 \geq  c_n'n!^{\frac{2}{n}}2^{{1}-\frac{2}{n}(\lfloor\frac{n}{2}\rfloor+1)}\int_{\sphn}\rho^{\frac{n+2}{n}}(x)dx,
\end{equation*}
$\rho(x)\defeq\sum_{j=1}^N|\psi_j(x)|^2,$
$c_n'$ is given by \eqref{eq_cn'}, and the integration is taken with respect to the normalized standard surface measure on the sphere.
\end{corollary}

We note that from the proof of Theorem \ref{teosn2}, if we consider the non-normalized surface measure of the sphere, the inequality above remains valid if we multiply the constant on the right by $(\sigma_{\sphn})^{\frac{2}{n}}$.

 For comparison, we cite the following result, which can be found in \cite{Kow}, but was adapted to consider the normalized measure on the sphere.

\begin{theorem}\label{teoLiebsn} 
Let $n\geq2$.  
    For any orthonormal family $\{\psi_j\}_{j=1}^N\subset H^{1}(\sphn)$ and $\int_{\sphn}\psi_j=0$ for all $j$, we have that  
    \begin{equation}\label{eqLieb}
        \sum_{j=1}^N\|\nabla \psi_j\|^2_{L^2} \geq  \frac{n!^{\frac{2}{n}}}{(n+4)}\left(\frac{n}{n+2}\right)^{\frac{n+2}{n}}\int_{\mathbb{S}^n}\rho(x)^{\frac{n+2}{n}}dx,
    \end{equation}
    where
$\rho(x)\defeq\sum_{j=1}^N|\psi_j(x)|^2,$
 and the integration is taken with respect to the normalized standard surface measure on the sphere.
\end{theorem}

Note that the constant obtained in Corollary \ref{coro_scalar} grows faster as $n\to \infty$, and in fact is greater than the constant 
from Theorem \ref{teoLiebsn} for $n\geq 5$. Therefore for $n\geq 5$ it provides a better estimate for the best constant for the usual Lieb-Thirring inequality on the sphere.

In summary:
\begin{theorem}\label{teoLiebsn_improved} 
    For any orthonormal family $\{\psi_j\}_{j=1}^N\subset H^{1}(\sphn)$ and $\int_{\sphn}\psi_j=0$ for all $j$, we have that  
    \begin{equation}\label{eqLieb_improved}
       \int_{\mathbb{S}^n}\rho(x)^{\frac{n+2}{n}}dx\leq (c_n')^{-1}n!^{-\frac{2}{n}}2^{\frac{2}{n}(\lfloor\frac{n}{2}\rfloor+1)-1} \sum_{j=1}^N\|\nabla \psi_j\|^2_{L^2} 
    \end{equation}
    where
$\rho(x)\defeq\sum_{j=1}^N|\psi_j(x)|^2,$ $c_n'$ is given by \eqref{eq_cn'}, 
 and the integration is taken with respect to the normalized standard surface measure on the sphere. In other words, 
 \begin{equation*}
k_{\sphn}\leq  (c_n')^{-1}n!^{-\frac{2}{n}}2^{\frac{2}{n}(\lfloor\frac{n}{2}\rfloor+1)-1}(\sigma_{\sphn})^{-\frac{2}{n}}.     
 \end{equation*}

\end{theorem}
In the table below we compare the estimates for the best Lieb-Thirring constant obtained in Theorems \ref{teoLiebsn} and \ref{teoLiebsn_improved} for a few values of $n$, but considering the non-normalized surface measure of the sphere.

\begin{table}[h!]
    \centering
   \begin{tabular}{||c|c| c||} 
 \hline
 $k_{\sphn}\leq $& Theorem \ref{teoLiebsn}& Theorem \ref{teoLiebsn_improved} \\ [0.5ex] 
 \hline\hline
$n=4$&0.5847726...&0.7033721...\\
\hline
$n=5$&0.5377363...&0.3788866...\\
\hline
$n=6$&0.5100914...&0.3054612...\\
\hline
$n=7$&0.4920770...&0.2100443...\\
\hline
$n=8$&0.4795110...&0.1854909...\\
 \hline
\end{tabular}
    \caption{Approximate values for the upper bounds for $k_{\sphn}$ (with respect to the  non-normalized surface measure of the sphere) according to Theorems \ref{teoLiebsn} and \ref{teoLiebsn_improved}.}
    \label{tab:placeholder}
\end{table}
 We can observe that Theorem \ref{teoLiebsn_improved} provides a better upper bound for $n\geq 5$, improving the currently known upper bounds for $k_{\sphn}$ for this range of values for $n$.

\section{Proofs of the main results}\label{section_proofs}

\begin{proof}[Proof of Theorem \ref{teos2}]
First recall (see for instance \cite{Camporesi}) that the set of non-constant normalised eigenfunctions $\{y_k^{j\pm}\}_{j,k}$ for the Dirac operator on $\mathbb{S}^2$ have corresponding eigenvalues 
\begin{equation*}
    \lambda_k^\pm=\pm(k+1),\,k\in\N_0.
\end{equation*}
These eigenvalues have multiplicity
\begin{equation*}
    2(k+1),\,k\in\N_0,
\end{equation*}
each and we denote by $y^{j\pm}_{k}$ the corresponding eigenfunctions.  These eigenfunctions form an orthonormal basis for the set of functions with zero mean.
Moreover, the following identity also holds, for every $x\in\mathbb{S}^2$:
\begin{equation}\label{addition_formula_s2_1}
   \sum_{j=1}^{2(k+1)}|y^{j-}_{k}(x)|^2=\sum_{j=1}^{2(k+1)}|y^{j+}_{k}(x)|^2=2(k+1).
\end{equation}
We note that this identity differs slightly from the one present in  \cite{Camporesi} as we are considering the normalised surface measure on the sphere.

 Next, we introduce the following notation for labeling the non-constant eigenfunctions and corresponding eigenvalues of the projection onto the positive spectrum of the Dirac operator with a single subscript counting multiplicities:
    \begin{equation*}
        \{y_j\}_{j=1}^\infty=\left\{y_{k}^{j+},\dots\right\},
    \end{equation*}
    and
    \begin{equation*}
        (\lambda_j)_{j=1}^\infty =((k+1),\dots),
    \end{equation*}
    where the positive eigenvalue $(k+1)$ is repeated $2(k+1)$ times.
    For $E\geq 0$, define the spectral projections:
    \begin{equation*}
        P_E=\sum_{\lambda_j<E}\langle\cdot,y_j\rangle y_j
    \end{equation*}
    and
    \begin{equation*}
        P_E^\perp=\sum_{\lambda_j\geq E}\langle\cdot,y_j\rangle y_j.
    \end{equation*}
    Then 
\begin{equation}\label{lapnoncte2_1}
        {\Lambda^+(\slashed D)}_0 =\sum_{j=1}^\infty \lambda_j\langle\cdot,y_j\rangle y_j,
    \end{equation}
    where $ {\Lambda^+(\slashed D)}_0$ denotes the restriction of $ \Lambda^+(\slashed D)$ to the invariant subspace of functions orthogonal to constants, that is 
    \begin{equation*}
         {\Lambda^+(\slashed D)}_0=P_{\lambda_1}^\perp\circ\Lambda^+(\slashed D)\circ P_{\lambda_1}^\perp.
    \end{equation*}
    Next, notice that
    \begin{align*}
        \sum_{j=1}^\infty\lambda_j a_j&=(\lambda_1-0)\sum_{j=1}^\infty a_j+(\lambda_2-\lambda_1)\sum_{j=2}^\infty a_j +\dots=\int_0^\infty\sum_{\lambda_j\geq E}a_j dE,
    \end{align*}
therefore we have the spectral decomposition 
    \begin{align}\label{eq1_2_1}
        \Lambda^+(\slashed D)_0 &=\int_0^\infty\sum_{\lambda_j\geq E}\langle\cdot,y_j\rangle y_jdE=
        \int_0^\infty P_E^\perp dE.
    \end{align}
In view of \eqref{lapnoncte2_1} and \eqref{eq1_2_1}, for $\psi=(\psi_1,\dots,\psi_N)$ we have 
    \begin{align*}
        \langle\Lambda^+(\slashed D)\psi,\psi\rangle=\langle\Lambda^+(\slashed D)_0\psi,\psi\rangle&=\int_{0}^\infty\langle P_E^\perp\psi,\psi\rangle dE\\
        &=\int_0^\infty \|P_E^\perp \psi\|^2dE\\
        &=\sum_{j=1}^N \int_0^\infty\int_{\sph}|P_E^\perp \psi_j(x)|^2dxdE\\
        &=\sum_{j=1}^N\int_{\sph}\int_0^\infty
        |P_E^\perp \psi_j(x)|^2dEdx.
    \end{align*}
     Denoting by $\Gamma_\psi$ the orthogonal projection:
    \begin{equation*}
        \Gamma_\psi=\sum_{j=1}^N\langle\cdot,\psi_j\rangle\psi_j,
    \end{equation*}
    we obtain
    \begin{align}\label{eqgradient2_1}
\langle\Lambda^+(\slashed D)\psi,\psi\rangle&=\int_{\sph}\int_0^\infty\rho_{P_E^\perp\Gamma_\psi P_E^\perp}(x)dEdx,
    \end{align}
    where
    \begin{equation*}
        \rho_{P_E^\perp\Gamma_\psi P_E^\perp}(x)\defeq\sum_{j=1}^N|P_E^\perp \psi_j(x)|^2.
    \end{equation*}
    Now let $B$ be a neighbourhood around $x_0\in \sph$, with measure $|B|\leq 1$, and let $\chi_B$ be the corresponding characteristic function. Then
    \begin{align}\label{eq22_1}
        \left(\int_B\rho(x)dx\right)^{\frac{1}{2}}&=\|\Gamma_\psi\chi_{B}\|_{HS}\notag\\
        &\leq \|\Gamma_\psi P_E\chi_B\|_{HS}+\|\Gamma_\psi P_E^\perp\chi_B\|_{HS}\notag\\
        &=\|\Gamma_\psi P_E \chi_B\|_{HS}+\left(\int_B\rho_{P_E^\perp\Gamma_\psi P_E^\perp}(x)dx\right)^{\frac{1}{2}},
    \end{align}
    where $\|\cdot\|_{HS}$ denotes the usual Hilbert-Schmidt norm for operators.
    Since $\|\Gamma_\psi\|=1$ and  both $P_E$ and $\chi_B$ are projections, we conclude that:
    \begin{align}
        \|\Gamma_\psi P_E \chi_B\|_{HS}^2&\leq \|P_E\chi_B\|_{HS}^2\notag\\
&=\sum_{\lambda_j<E}\int_{\mathbb{S}^{2}}|y_j(x)|^2\chi_B(x)dx\notag\\
        &=\sum_{\substack{(k+1)<E\\k\in\N_0}}\int_{\mathbb{S}^{2}}\sum_{j=1}^{2(k+1)}{|y_{k}^{j+}(x)|^2}\chi_B(x)dx\notag\\
        &\stackrel{\eqref{addition_formula_s2_1}}{=}|B|\sum_{\substack{(k+1)<E\\k\in\N_0}}2(k+1)\notag\\
        &= 
            |B|(m^2+m),\,\quad\text{for }\quad m< E\leq m+1,\ m\in\N_0,\label{ineq_s21}
    \end{align}
    where the last equality follows from the fact that for $m<E\leq m+1$ we have that
    \begin{align*}
        \sum_{\substack{(k+1)<E\\k\in\N_0}}2(k+1)&= 2\sum_{\substack{k'=1}}^{m}k'\\
        &=2\left(\frac{(m+1)m}{2}\right)\\
        &=m^2+m.
    \end{align*}
    Next, let
    \begin{equation}
         C(E) \defeq m^2+m,\quad\text{for }\quad  m< E\leq m+1,\ m\in\N_0.
    \end{equation}
    Substituting inequality \eqref{ineq_s21} in \eqref{eq22_1}, dividing both sides by $|B|^{\frac{1}{2}}$, and letting $|B|\to 0$, we obtain
    \begin{equation*}
        \rho(x_0)^{1/2}\leq C(E)^{1/2}+\rho_{P_E^\perp\Gamma P_E^\perp}(x_0)^{1/2}, 
    \end{equation*}
    for almost every $x_0\in \sph$. 
    Since $\rho_{P_E^\perp\Gamma P_E^\perp}\geq0$, this implies
    \begin{equation*}
        \rho_{P_E^\perp\Gamma P_E^\perp}(x_0)\geq (\rho(x_0)^{1/2}-C(E)^{1/2})_+^2,
    \end{equation*}
    for almost every $x_0\in \sph$, where the expression $(\rho(x_0)^{1/2}-C(E)^{1/2})_+$ denotes the positive part of this difference. Integrating this expression over $x_0\in \sph$ and applying \eqref{eqgradient2_1} yields
\begin{equation}\label{eqgradientfinal2_1}
         \langle\Lambda^+(\slashed D)\psi,\psi\rangle\geq \int_{\sph}\int_0^\infty(\rho(x)^{1/2}-C(E)^{1/2})_+^2 dEdx.
    \end{equation}
     We proceed to estimate the integral above as follows: for $\rho> 0$, let
    \begin{equation}\label{Irho2_1}
        I(\rho)\defeq\int_0^\infty(\rho^{1/2}-C(E)^{1/2})_+^2 dE.
    \end{equation}
For any fixed $\rho>0$, let $M\in \N_0$ be the unique non-negative integer such that
\begin{equation*}
    M^2+M\leq \rho<(M+1)^2+(M+1).
\end{equation*}
Then 
\begin{align}
    I(\rho)&=\sum_{m=0}^M\int_{m}^{m+1}(\rho^{1/2}-(m^2+m)^{1/2})^2dE\notag\\
    &=\sum_{m=0}^M[\rho-2\sqrt{\rho}\sqrt{m^2+m}+m^2+m]\notag\\
    &=(M+1)\rho-2\sqrt{\rho}\sum_{m=0}^M\sqrt{m^2+m}+\frac{1}{3}M(M+1)(M+2).\label{ineq_I(rho)_1}
\end{align}
Note that
\begin{equation*}
    \sqrt{(m+1)(m+2)}-\sqrt{m(m-1)}>2,
\end{equation*}
for $m\geq 1$, therefore
\begin{equation*}
    2\sqrt{m^2+m}<(m+1)\sqrt{m(m+2)}-m\sqrt{(m-1)(m+1)}.
\end{equation*}
Setting $g(m)\vcentcolon=(m+1)\sqrt{m(m+2)}$, the inequality above corresponds to $ 2\sqrt{m^2+m}<g(m)-g(m-1)$. Therefore
\begin{equation*}
    2\sum_{m=0}^M\sqrt{m^2+m}\leq g(M)=(M+1)\sqrt{M(M+2)},
\end{equation*}
so that
\begin{align*}
   2I(\rho)&\geq (M+1)\Big(2\rho-2\sqrt{\rho}\sqrt{M(M+2)}+\frac{2}{3}M(M+2)\Big)\\
   &=(M+1)\Big(\rho+\Big(\sqrt{\rho}-\sqrt{M(M+2)}\Big)^2-\frac{1}{3}M(M+2)\Big)\\
   &\geq (M+1)\Big(\rho-\frac{1}{3}M(M+2)\Big).
\end{align*}
Therefore
\begin{align*}
     \frac{I(\rho)}{\rho^{\frac{3}{2}}}&\geq  \frac{1}{2}(M+1)\rho^{-\frac{1}{2}}\Big(1-\frac{1}{3}M(M+2)\rho^{-1}\Big)=\vcentcolon f(\rho,M).   
\end{align*}
Differentiating $f$ with respect to $\rho$, we find that  for each fixed $M$, its unique critical point lies at $\rho=M^2+2M\in [M^2+M,(M+1)^2+M+1)$, and that this critical point corresponds to a local maximum. Therefore $f(\rho,M)$ attains its minimum over $[M^2+M,(M+1)^2+(M+1)]$ at one of the endpoints of this interval. 

Note that 
\begin{equation*}
    f(M^2+M,M)=\frac{2M+1}{6\sqrt{M(M+1)}},
\end{equation*}
while 
\begin{equation*}
    f((M+1)^2+M+1,M)=\frac{2M+3}{6\sqrt{(M+1)(M+2)}}.
\end{equation*}
Since the function $x\mapsto \frac{1}{6}\frac{2x+1}{\sqrt{x(x+1)}}$ is decreasing for $x>0$, we conclude that $f(\rho,M)$ attains its minimum over $[M^2+M,(M+1)^2+M+1]$ at $(M+1)^2+M+1$. By continuity we get that
\begin{equation*}
    \inf_{ \rho\in[M^2+M,(M+1)^2+M+1)}\frac{I(\rho)}{\rho^{\frac{3}{2}}}\geq \frac{2M+3}{6\sqrt{(M+1)(M+2)}}.
\end{equation*}
Finally, noting that $\frac{2M+3}{6\sqrt{(M+1)(M+2)}}$ is decreasing with respect to $M$, we conclude that 
\begin{equation}\label{ineq_I(rho)_final_1}
    \inf_{ \rho>0}\frac{I(\rho)}{\rho^{\frac{3}{2}}}\geq \lim_{M\to \infty}\frac{2M+3}{6\sqrt{(M+1)(M+2)}}=\frac{1}{3}.
\end{equation}
Since 
\begin{equation*}
     \sum_{j=1}^N\langle\Lambda^+(\slashed D)\psi_j,\psi_j\rangle\geq \int_{\sph}I(\rho(x))dx
\end{equation*}
by \eqref{eqgradientfinal2_1}, from \eqref{ineq_I(rho)_final_1} we conclude that 
\begin{equation*}
     \sum_{j=1}^N\langle\Lambda^+(\slashed D)\psi_j,\psi_j\rangle\geq\frac{1}{3}\int_{\sph}\rho^{\frac{3}{2}}(x)dx,
\end{equation*}
completing the proof.
\end{proof}

\begin{proof}[Proof of Theorem \ref{teosn}]
First recall (see for instance \cite{Camporesi}) that the non-constant normalised eigenfunctions $\{y_k^{j\pm}\}_{j,k}$ for the Dirac operator on $\mathbb{S}^n$ (which form an orthonormal basis for the set of functions with zero mean) have corresponding eigenvalues 
\begin{equation}
    \lambda_k^\pm=\pm\left(\frac{n}{2}+k\right),\,k\in\N_0.
\end{equation}
Each of these eigenvalues has multiplicity
\begin{equation}\label{eq_mk_def}
    m_k\defeq2^{\lfloor \frac{n}{2}\rfloor}\cdot\binom{k+n-1}{k},\,k\in\N_0.
\end{equation}
The following identity also holds for every $x\in\mathbb{S}^n$:
\begin{equation}\label{addition_formula_sn_1}
   \sum_{j=1}^{m_k}|y^{j-}_{k}(x)|^2=\sum_{j=1}^{m_k}|y^{j+}_{k}(x)|^2=m_k.
\end{equation}
As before we note that this identity differs slightly from the one present in  \cite{Camporesi} as we are considering the normalised surface measure on the sphere.

As in the proof of Theorem \ref{teos2}, consider  the following notation for labeling these eigenfunctions and corresponding eigenvalues of the projection onto the positive spectrum of the Dirac operator with a single subscript counting multiplicities:
    \begin{equation*}
        \{y_j\}_{j=1}^\infty=\left\{y_{k}^{j+},\dots\right\},
    \end{equation*}
    and
    \begin{equation*}
        (\lambda_j)_{j=1}^\infty =\left(\frac{n}{2}+k,\dots\right),
    \end{equation*}
    where the positive eigenvalue $\frac{n}{2}+k$ is repeated $m_k$ times.
  Analogous to the proof of Theorem \ref{teos2}, we then proceed to obtain
\begin{align}\label{eqgradient_1}
\sum_{j=1}^N\langle\Lambda^+(\slashed D)\psi_j,\psi_j\rangle&=\int_{\sphn}\int_0^\infty\rho_{P_E^\perp\Gamma_\psi P_E^\perp}(x)dEdx,
    \end{align}
    where
    \begin{equation*}
        \rho_{P_E^\perp\Gamma_\psi P_E^\perp}(x)\defeq\sum_{j=1}^N|P_E^\perp \psi_j(x)|^2.
    \end{equation*}    
    We also have that for a small neighbourhood $B$ around $x_0\in \sphn$, and $\chi_B$ its characteristic function, the following inequality holds:
    \begin{align}\label{ineq_sn_rho_projection}
        \left(\int_B\rho(x)dx\right)^{\frac{1}{2}}
        &\leq\|\Gamma_\psi P_E \chi_B\|_{HS}+\left(\int_B\rho_{P_E^\perp\Gamma_\psi P_E^\perp}(x)dx\right)^{\frac{1}{2}}.
    \end{align}
    Since both $\chi_B$ and $P_E$ are projections and $\|\Gamma_\psi\|=1$, we find that
    \begin{align}
        \|\Gamma_\psi P_E \chi_B\|_{HS}^2&\leq \|P_E\chi_B\|_{HS}^2\notag\\
&=\sum_{\lambda_j<E}\int_{\mathbb{S}^{n}}|y_j(x)|^2\chi_B(x)dx\notag\\
        &=\sum_{\substack{\frac{n}{2}+k<E\\k\in\N_0}}\int_{\mathbb{S}^{n}}\sum_{j=1}^{m_k}|y_{k}^{j+}(x)|^2\chi_B(x)dx\notag\\
&\stackrel{\eqref{addition_formula_sn_1}}{=}|B|\sum_{\substack{\frac{n}{2}+k<E\\k\in\N_0}}2^{\lfloor \frac{n}{2}\rfloor}\cdot\binom{k+n-1}{k}\notag\\
        &\leq 
        \begin{cases}
            |B|2^{\lfloor\frac{n}{2}\rfloor}\frac{1}{n!}(E+\frac{1}{2})^n,\qquad\,\text{if }E>\frac{n}{2},\\
            0,\,\quad\quad\qquad\qquad\qquad\qquad\text{ otherwise,}
        \end{cases}
        \label{ineq_sn_projection}
    \end{align}
     where the last inequality is justified as follows. Note that the largest integer strictly less that $E-\frac{n}{2}$ is given by $\lceil E-\frac{n}{2}\rceil -1$. So for $E>\frac{n}{2}$ we have that
\begin{align*}
     \sum_{\substack{\frac{n}{2}+k<E\\k\in\N_0}}\binom{k+n-1}{k}&=\sum_{k=0}^{\lceil E-\frac{n}{2}\rceil -1}\binom{k+n-1}{k}=
     \binom{\lceil E-\textstyle{\frac{n}{2}}\rceil -1+n}{\lceil E-\frac{n}{2}\rceil -1}\\
     &=\frac{1}{n!}((\lceil E-\textstyle{\frac{n}{2}}\rceil -1+n)\cdots (\lceil E-\frac{n}{2}\rceil ))\\
     &\leq\frac{1}{n!}\left(\frac{n(\lceil E-\frac{n}{2}\rceil)+\frac{n(n-1)}{2}}{n}\right)^n\\
     &=\frac{1}{n!}\left({\lceil E-\textstyle{\frac{n}{2}}\rceil+\frac{(n-1)}{2}}\right)^n\\
     &<\frac{1}{n!}\left({ E-\frac{n}{2}+1+\frac{(n-1)}{2}}\right)^n\\
     &=\frac{1}{n!}\left({ E+\frac{1}{2}}\right)^n,
\end{align*}
where we applied the AM-GM inequality in the third line. Next, let
    \begin{equation}\label{eq_def_C(E)}
         C(E) \defeq \begin{cases}
            (E+\frac{1}{2})^n,\,\quad\quad\qquad\text{if }E>\frac{n}{2}\\
            0,\,\qquad\qquad\qquad\quad\text{ otherwise.}
        \end{cases}
    \end{equation}

Substituting inequality \eqref{ineq_sn_projection} in \eqref{ineq_sn_rho_projection}, dividing both sides by $|B|^{\frac{1}{2}}$, and letting $|B|\to 0$, we obtain
    \begin{equation*}
        \rho(x_0)^{1/2}\leq (K_n)^{\frac{1}{2}}C(E)^{1/2}+\rho_{P_E^\perp\Gamma P_E^\perp}(x_0)^{1/2}, 
    \end{equation*}
    for almost every $x_0\in \sphn$ where
    \begin{equation}\label{K_n_2}
    K_n=2^{\lfloor\frac{n}{2}\rfloor}\frac{1}{n!}.
    \end{equation}
    Since $\rho_{P_E^\perp\Gamma P_E^\perp}\geq0$, this implies
    \begin{equation*}
        \rho_{P_E^\perp\Gamma P_E^\perp}(x_0)\geq (\rho(x_0)^{1/2}-(K_n)^{\frac{1}{2}}C(E)^{1/2})_+^2,
    \end{equation*}
    for almost every $x_0\in \sphn$, where the expression $(\rho(x_0)^{1/2}-(K_n)^{\frac{1}{2}}C(E)^{1/2})_+$ denotes the non-negative part of the difference. Integration over $x_0\in \sphn$ and equality \eqref{eqgradient_1} implies
    \begin{equation}\label{eqgradientfinal_2}
         \sum_{j=1}^N\langle\Lambda ^+(\slashed D)\psi_j,\psi_j\rangle\geq \int_{\sphn}\int_0^\infty(\rho(x)^{1/2}-(K_n)^{\frac{1}{2}}C(E)^{1/2})_+^2 dEdx.
    \end{equation}
    Therefore, what is left is to estimate this last integral. We proceed as follows: for $\rho> 0$, define
    \begin{equation}\label{Irho_2}
        I(\rho)\defeq\int_0^\infty(\rho^{1/2}-(K_n)^{\frac{1}{2}}C(E)^{1/2})_+^2 dE.
    \end{equation}
     Then for $\rho_n=\rho/K_n$ and $I_0(\rho_n)\defeq \int_0^\infty({\rho_n}^{1/2}-C(E)^{1/2})_+^2 dE$, we have that
        \begin{equation}
        I_0(\rho_n)\cdot K_n=\int_0^\infty({\rho}^{1/2}K_n^{-1/2}-C(E)^{1/2})_+^2 dE\cdot K_n=I(\rho).
    \end{equation}
Then since $C(E)=0$ for $0<E<\frac{n}{2}$ and $E\mapsto (E+\frac{1}{2})^n$ is increasing, for $\rho> (\frac{n}{2}+\frac{1}{2})^n=(\frac{n+1}{2})^n$ we have that
    \begin{align*}
        I_0(\rho_n)&\geq \int_0^{\frac{n}{2}}({\rho_n}^{1/2}-0)^2dE+\int_{\frac{n}{2}}^{+\infty}\left({\rho_n}^\frac{1}{2}-{\left(E+\textstyle{\frac{1}{2}}\right)^\frac{n}{2}}\right)_+^2dE\\
        &=\frac{n}{2}\rho_n+\int_{\frac{n}{2}}^{\sqrt[n]{{\rho_n}}-\textstyle{\frac{1}{2}}}(\rho_n^{1/2}-(E+\textstyle{\frac{1}{2}})^{\frac{n}{2}})^2dE\\
        &=\frac{n}{2}\rho_n+\left(\frac{\sqrt{\rho_n}(n+1)^{\frac{n}{2}+1}}{2^{\frac{n}{2}-1}(n+2)}+\frac{\rho_n^{\frac{n+1}{n}}n^2}{n^2+3n+2}-\left(\frac{n+1}{2}\right)\rho_n-\frac{(n+1)^n}{2^{n+1}}\right)\\
        &=\frac{\sqrt{\rho_n}(n+1)^{\frac{n}{2}+1}}{2^{\frac{n}{2}-1}(n+2)}+\frac{\rho_n^{\frac{n+1}{n}}n^2}{n^2+3n+2}-\frac{\rho_n}{2}-\frac{(n+1)^n}{2^{n+1}}=\vcentcolon I_1(\rho_n).
\end{align*}
 But notice that 
\begin{align*}
    \frac{d}{dx}\frac{I_1(x)}{x^{\frac{n+1}{n}}} &= \frac{(n+1)^{n+1}}{n \cdot 2^{n+1}} x^{-\frac{2n+1}{n}} + \frac{1}{2n} x^{-\frac{n+1}{n}} - \frac{(n+1)^{\frac{n}{2}+1}}{n \cdot 2^{\frac{n}{2}}} x^{-\frac{3n+2}{2n}}\\
    &=\frac{1}{n}x^{-\frac{3n+2}{2n}}\left(\frac{(n+1)^{n+1}}{2^{n+1}}x^{-\frac{1}{2}}+\frac{1}{2}x^{\frac{1}{2}}-\frac{(n+1)^{\frac{n+2}{2}}}{2^{\frac{n}{2}}}\right)\\
    &=\vcentcolon f_n(x).
\end{align*}
Hence the critical points of $\frac{I_1(x)}{x^{\frac{n+1}{n}}}$ are given by the roots of $f_n(x)$, for $x>0$. By performing the change of variables $x={u}^2$, finding the positive roots of $f_n(x)=0$ amounts to solving the quadratic equation in $u$ given by:
\begin{equation}\label{eq_quadratic_1}
    2^n u^2-2^{\frac{n+2}{2}}(n+1)^{\frac{n+2}{2}}u+(n+1)^{n+1}=0.
\end{equation}
The roots of this equation are:
\begin{equation*}
    u_{\pm}=\left(\frac{n+1}{2}\right)^{\frac{n}{2}}\sqrt{n+1}(\sqrt{n+1}\pm\sqrt{n}),
\end{equation*}
corresponding to the critical points
\begin{equation*}
    x_{\pm}=\left(\frac{n+1}{2}\right)^{{n}}({n+1})(\sqrt{n+1}\pm\sqrt{n})^{2}.
\end{equation*}
Next, note that 
\begin{align*}
    f_n(u^2)=\frac{1}{n}u^{-\frac{3n+2}{n}}\frac{2^n u^2-2^{\frac{n+2}{2}}(n+1)^{\frac{n+2}{2}}u+(n+1)^{n+1}}{2^{n+1}u},
\end{align*}
hence its sign is also determined by the previous upward pointing quadratic \eqref{eq_quadratic_1}. Substituting $u=\sqrt{(\frac{n+1}{2})^n}$ in \eqref{eq_quadratic_1}, we obtain
\begin{align*}
    (n+1)^n-(n+1)^{n+1}<0,
\end{align*}
so $f_n({\textstyle{(\frac{n+1}{2})^n}})<0$ and $(\frac{n+1}{2})^n$ lies between the roots $x_{\pm}$. We conclude that $f_n$ is negative for $(\frac{n+1}{2})^n<x<x_+$ and positive for $x>x_+$, hence the point $x_+$ corresponds to a minimum for $\frac{I_1(x)}{x^{\frac{n+1}{n}}}$ for $x>(\frac{n+1}{2})^n$.

Therefore, 
\begin{align*}
    \inf_{\rho_n >(\textstyle{\frac{n+1}{2}})^n}\frac{I_0(\rho_n)}{\rho_n^{\frac{n+1}{n}}}&\geq \inf_{\rho_n >(\textstyle{\frac{n+1}{2}})^n}\frac{I_1(\rho_n)}{\rho_n^{\frac{n+1}{n}}}\\
    &=\frac{I_1(x_+)}{x_+^{\frac{n+1}{n}}}\\
    &=\frac{(\left(\frac{n+1}{2}\right)^{{n}}({n+1})(\sqrt{n+1}+\sqrt{n})^{2})^{\frac{1}{2}-\frac{n+1}{n}}(n+1)^{\frac{n}{2}+1}}{2^{\frac{n}{2}-1}(n+2)}+\frac{n^2}{n^2+3n+2}\\
    &\phantom{=}-\frac{(\left({\textstyle{\frac{n+1}{2}}}\right)^{{n}}({n+1})(\sqrt{n+1}+\sqrt{n})^{2})^{-\frac{1}{n}}}{2}\\
    &\phantom{=}-\frac{(n+1)^n}{2^{n+1}}(\left({\textstyle{\frac{n+1}{2}}}\right)^{{n}}({n+1})(\sqrt{n+1}+\sqrt{n})^2)^{-\frac{n+1}{n}}\\
    &= \frac{n^2}{n^2+3n+2}-\frac{(\sqrt{n+1}-\sqrt{n})^{\frac{2}{n}}}{{(n+1)^{\frac{n+1}{n}}}}\frac{n(2\sqrt{n(n+1)}-n)}{n^2+3n+2}\\
    &=\vcentcolon c_n.
\end{align*}

On the other hand, for $0<\rho_n\leq (\frac{n+1}{2})^n$, we have that
\begin{align*}
    I_0(\rho_n)/(\rho_n)^{\frac{n+1}{n}}=\frac{n}{2}\rho_n^{-\frac{1}{n}}\geq \frac{n}{2}\left(\frac{n+1}{2}\right)^{-1}=\frac{n}{n+1}.
\end{align*}
Note that $c_n<\frac{n}{n+1}$, for every $n\in\mathbb{N}$. Indeed, this is equivalent to proving that $c_n-\frac{n}{n+1}<0$, or equivalently, that
\begin{equation*}
    S(n)<\frac{n}{n+1}-\frac{n^2}{n^2+3n+2}=\frac{2n}{n^2+3n+2},
\end{equation*}
where
\begin{equation*}
    S(n)=-\frac{(\sqrt{n+1}-\sqrt{n})^{\frac{2}{n}}}{{(n+1)^{\frac{n+1}{n}}}}\frac{n(2\sqrt{n(n+1)}-n)}{n^2+3n+2}.
\end{equation*}
Therefore it is sufficient to prove that $S(n)$ is always negative. But this is clear since  $\sqrt{n(n+1)}>n$, for every $n\in\N$.

Consequently,  we have that 
\begin{align}\label{M_n_2}
    I_0(\rho_n)/(\rho_n)^{\frac{n+1}{n}}\geq\min\left\{c_n,\frac{n}{n+1}\right\}=c_n,
\end{align}
for every $\rho_n>0$, from which it follows that
\begin{align}\label{eq_I(rho)_relations}
    \frac{I(\rho)}{\rho^{\frac{n+1}{n}}}=\frac{I_0(\rho_n)\cdot K_n}{(\rho_n)^{\frac{n+1}{n}}\cdot (K_n)^{\frac{n+1}{n}}}\geq c_n\cdot (K_n)^{-\frac{1}{n}},
\end{align}
for every $\rho>0$.
Therefore, from equations \eqref{eqgradientfinal_2}, \eqref{Irho_2}, \eqref{M_n_2}, \eqref{eq_I(rho)_relations}  and the definition of $K_n$, we obtain that
\begin{equation*}
    \sum_{j=1}^N\langle\Lambda^+(\slashed D)\psi_j,\psi_j\rangle\geq c_n\cdot 
     {{n!}^{\frac{1}{n}}2^{-\frac{1}{n}\lfloor \frac{n}{2}\rfloor}}\int_{\sphn}\rho^{\frac{n+1}{n}}(x)dx,
\end{equation*}
concluding the proof.
\end{proof}

\begin{proof}[Proof of Theorem \ref{teosn2}]
The proof follows an argument analogous to the proof of Theorem \ref{teosn}, by performing the following adaptations:

Consider the following alternative notation for labeling the non-constant eigenfunctions and corresponding eigenvalues of the Dirac operator with a single subscript counting multiplicities:   
  \begin{equation*}
        \{y_j^{\pm}\}_{j=1}^\infty=\left\{y_{k}^{j\pm},\dots\right\},
    \end{equation*}
    and
     \begin{equation*}
        (\lambda_j)_{j=1}^\infty =\left(\frac{n}{2}+k,\dots\right),
    \end{equation*}
    where the positive eigenvalue $\frac{n}{2}+k$ is repeated $m_k$ times, and $m_k$ is given as in \eqref{eq_mk_def}.
Once again, for $E\geq 0$ consider the corresponding spectral projections:
    \begin{equation*}
        P_E=\sum_{\lambda_j<E}[\langle\cdot,y_j^+\rangle y_j^++\langle\cdot,y_j^-\rangle y_j^-]
    \end{equation*}
    and
    \begin{equation*}
        P_E^\perp=\sum_{\lambda_j\geq E}[\langle\cdot,y_j^+\rangle y_j^++\langle\cdot,y_j^-\rangle y_j^-],
    \end{equation*}
    so that
\begin{equation}\label{lapnoncte2_2}
        \dot{\slashed D} =\sum_{j=1}^\infty [\lambda_j\langle\cdot,y_j^+\rangle y_j^+-\lambda_j\langle\cdot,y_j^-\rangle y_j^-],
    \end{equation}
    where $ \dot{\slashed D}=P_{\lambda_1}^\perp\circ\slashed D\circ P_{\lambda_1}^\perp.$ denotes the restriction of $ \dot{\slashed D}$ to the invariant subspace of functions orthogonal to constants.
    
Thus, analogously as before, we obtain 
\begin{align*}
    \langle \slashed D\psi,\slashed D \psi \rangle=2\sum_{j=1}^N\int_{\sphn}\int_0^\infty E\cdot|P_E^\perp \psi_j(x)|^2dEdx.
    \end{align*}
Once again, with $\Gamma_\psi\vcentcolon=\sum_{j=1}^N\langle\cdot,\psi_j\rangle\psi_j,$
    we obtain
    \begin{align}\label{eq_dirac_norm}
\|\slashed D\psi\|_{L^2(\R_{0,d})}^2 &=2\int_{\sph}\int_0^\infty E\cdot\rho_{P_E^\perp\Gamma_\psi P_E^\perp}(x)dEdx,
    \end{align}
    where
 $\rho_{P_E^\perp\Gamma_\psi P_E^\perp}(x)\defeq\sum_{j=1}^N|P_E^\perp \psi_j(x)|^2,$ and for any small neighbourhood $B$ centered at $x_0\in \sph$, we have that
    \begin{align}\label{eq22_2}
        \left(\int_B\rho(x)dx\right)^{\frac{1}{2}}&=\|\Gamma_\psi\chi_{B}\|_{HS}\notag\\
        &\leq \|\Gamma_\psi P_E\chi_B\|_{HS}+\|\Gamma_\psi P_E^\perp\chi_B\|_{HS}\notag\\
        &=\|\Gamma_\psi P_E \chi_B\|_{HS}+\left(\int_B\rho_{P_E^\perp\Gamma_\psi P_E^\perp}(x)dx\right)^{\frac{1}{2}},
    \end{align}
    where $\chi_B$ denotes the characteristic function of $B$.
   Then, similarly to  the previous proof, we obtain
    \begin{align}
        \|\Gamma_\psi P_E \chi_B\|_{HS}^2&\leq \|P_E\chi_B\|_{HS}^2\notag\\
&=\sum_{\lambda_j<E}\int_{\mathbb{S}^{n}}(|y_j^+(x)|^2+|y_j^-(x)|^2)\chi_B(x)dx\notag\\
        &=\sum_{\substack{\frac{n}{2}+k<E\\k\in\N_0}}\int_{\mathbb{S}^{n}}\sum_{j=1}^{m_k}(|y_{k}^{j+}(x)|^2+|y_k^{j-}(x)|^2)\chi_B(x)dx\notag\\
&\stackrel{\eqref{addition_formula_sn_1}}{=}|B|\sum_{\substack{\frac{n}{2}+k<E\\k\in\N_0}}2^{(\lfloor \frac{n}{2}\rfloor+1)}\cdot\binom{k+n-1}{k}\notag\\
        &\leq 
        \begin{cases}
            |B|2^{(\lfloor\frac{n}{2}\rfloor+1)}\frac{1}{n!}(E+\frac{1}{2})^n,\qquad\,\text{if }E>\frac{n}{2},\\
            0,\,\quad\quad\qquad\qquad\qquad\qquad\text{ otherwise.}
        \end{cases}
        \label{ineq_sn_projection_2}
    \end{align}

Proceeding as in the proof of Theorem \ref{teosn}, we obtain
    \begin{equation*}
        \rho_{P_E^\perp\Gamma P_E^\perp}(x_0)\geq (\rho(x_0)^{1/2}-(K_n')^{\frac{1}{2}}C(E)^{1/2})_+^2,
    \end{equation*}
    for almost every $x_0\in \sphn$, where the $K_n'=2^{(\lfloor\frac{n}{2}\rfloor+1)}\frac{1}{n!}$ and $E\mapsto C(E)$ is given by \eqref{eq_def_C(E)}. Integrating the expression above, equality \eqref{eq_dirac_norm} then implies that
\begin{equation}\label{eqgradientfinal_2_2}
         \sum_{j=1}^N\langle\slashed D\psi_j,\slashed D\psi_j\rangle\geq 2\int_{\sphn}\int_0^\infty E\cdot (\rho(x)^{1/2}-(K_n')^{\frac{1}{2}}C(E)^{1/2})_+^2 dEdx.
    \end{equation}
    In order to estimate this integral, let
     \begin{equation}\label{Irho_22}
        I(\rho)\defeq\int_0^\infty E\cdot (\rho^{1/2}-(K_n')^{\frac{1}{2}}C(E)^{1/2})_+^2 dE.
    \end{equation}
     Then putting $\rho_n=\rho/K_n'$ it is enough to estimate $I_0(\rho_n)\defeq  \int_0^\infty E\cdot({\rho_n}^{1/2}-C(E)^{1/2})_+^2 dE$.

Since $C(E)=0$ for $0<E<\frac{n}{2}$, and $E\mapsto (E+\frac{1}{2})^n$ is increasing, for $\rho_n> (\frac{n}{2}+\frac{1}{2})^n=(\frac{n+1}{2})^n$ we have that
    \begin{align*}
        I_0(\rho_n)&\geq \int_0^{\frac{n}{2}}E\cdot ({\rho_n}^{\frac{1}{2}}-0)^2dE+\int_{\frac{n}{2}}^{+\infty}E\cdot \left({\rho_n}^{\frac{1}{2}}-{\left(E+\textstyle{\frac{1}{2}}\right)^\frac{n}{2}}\right)_+^2dE\\
        &=\frac{n^2}{8}\rho_n+\int_{\frac{n}{2}}^{\sqrt[n]{{\rho_n}}-\textstyle{\frac{1}{2}}}E\rho_n-2\rho_n^{\frac{1}{2}}E(E+{\textstyle{\frac{1}{2}}})^{\frac{n}{2}}+E(E+\frac{1}{2})^{n}dE\\
        &=\frac{\rho_n}{8} + \frac{n^2\rho_n}{2(n+2)}\left(\frac{\rho_n^{2/n}}{n+4} - \frac{\rho_n^{1/n}}{n+1}\right) \\
        &+ 2\sqrt{\rho_n}\left(\frac{2}{n+4}\left(\frac{n+1}{2}\right)^{\frac{n+4}{2}} - \frac{1}{n+2}\left(\frac{n+1}{2}\right)^{\frac{n+2}{2}}\right) \\
        &- \frac{(n+1)^n(n^2+n-1)}{2^{n+2}(n+2)}\\
        &=\vcentcolon I_1(\rho_n).
\end{align*}
Setting 
\begin{align*}
    A_n&=\frac{2}{n+4}\left(\frac{n+1}{2}\right)^{\frac{n+4}{2}} - \frac{1}{n+2}\left(\frac{n+1}{2}\right)^{\frac{n+2}{2}},\\
    B_n&=\frac{(n+1)^n(n^2+n-1)}{2^{n+2}(n+2)},
\end{align*}
for every $n\in\N$, we obtain that 
\begin{align*}
    \frac{I_1(x)}{x^{\frac{n+2}{n}}}=\frac{n^2}{2(n+2)(n+4)}-\frac{n^2}{2(n+1)(n+2)}x^{-\frac{1}{n}}+\frac{1}{8}x^{-\frac{2}{n}}+2A_n x^{-\frac{n+4}{2n}}-B_n x^{-\frac{n+2}{n}},
\end{align*}
for every $x>0$. Let $y=x^{-\frac{1}{n}}$. Then 
\begin{align*}
     \frac{I_1(x)}{x^{\frac{n+2}{n}}}= \frac{I_1(y(x))}{y(x)^{\frac{n+2}{n}}}&=\frac{1}{8}y^2-\frac{n^2}{2(n+1)(n+2)}y+\frac{n^2}{2(n+2)(n+4)}+2A_n y^{\frac{n+4}{2}}-B_n y^{n+2}\\
     &=q(y)+R(y),
\end{align*}
where $q(y)$ corresponds to the quadratic
\begin{equation*}
    q(y)=\frac{1}{8}y^2-\frac{n^2}{2(n+1)(n+2)}y+\frac{n^2}{2(n+2)(n+4)},
\end{equation*}
while $R(y)$ corresponds to the remainder
\begin{equation*}
    R(y)=2A_n y^{\frac{n+4}{2}}-B_n y^{n+2}.
\end{equation*}
Note that $q(y)$  attains a global minimum at $y_0=\frac{2n^2}{(n+1)(n+2)}$. For $n=1$,  $y_0=1/3$ corresponds to $x=3>((1+1)/2)^1=1$. For $n\geq 2$, note that $x>(\frac{n+1}{2})^n$ if and only if $y\in (0,\frac{2}{n+1})$. Moreover, 
\begin{equation*}
    \frac{2n^2}{(n+1)(n+2)}=\frac{2}{(n+1)}\frac{n^2}{(n+2)}\geq \frac{2}{(n+1)},
\end{equation*}
therefore  $q(y(x))$ has a global minimum for $x_0<(\frac{n+1}{2})^n$, and since $q$ is an upward pointing parabola, we conclude that 
\begin{equation*}
    \inf_{x>(\frac{n+1}{2})^n}q(y(x))=\begin{cases}
           q\Big(y\Big({(\textstyle{\frac{n+1}{2}}})^n\Big)\Big)= \frac{n^4-6n^2+6n+8}{2n^4+16n^3+42n^2+44n+16},&\text{ if }n\geq 2,\\
           q(1/3)=\frac{7}{360}&\text{ if }n=1.
    \end{cases} 
\end{equation*}
Next note that $R(y)>0$ for every $x>(\frac{n+1}{2})^n$. Indeed, this is equivalent to $2A_n-B_nx^{-\frac{1}{2}}>0$ for all such $x$. But since $x^{-\frac{1}{2}}$ is decreasing and 
\begin{align*}
   2A_n-B_n\left(\Big(\frac{n+1}{2}\Big)^n\right)^{-\frac{1}{2}}&=2\left(\frac{n+1}{2}\right)^{\frac{n+2}{2}}\frac{n^2+2n-2}{(n+2)(n+4)} -\frac{(n+1)^{\frac{n}{2}}(n^2+n-1)}{2^{\frac{n}{2}+2}(n+2)}\\
   &=\left(\frac{n+1}{2}\right)^{\frac{n}{2}}\frac{1}{(n+2)}\left(\frac{(n+1)(n^2+2n-2)}{n+4}-\frac{n^2+n-1}{4}\right)\\
   &=\left(\frac{n+1}{2}\right)^{\frac{n}{2}}\frac{1}{(n+2)}\left(\frac{3n^3+7n^2-3n-4}{4(n+4)}\right),
\end{align*}
and $7n^2\geq 3n+4$ for every $n\geq 2$, we conclude that $R(y)>0$ for every $n\in\N$ (the case $n=1$ is also true by inspection). Consequently, 
\begin{equation}\label{lower_bound_norm}
   \inf_{x>(\frac{n+1}{2})^n}\frac{I_1(x)}{x^{\frac{n+2}{n}}}\geq  \inf_{x>(\frac{n+1}{2})^n}q(y(x))=\begin{cases}
            \frac{n^4-6n^2+6n+8}{2n^4+16n^3+42n^2+44n+16},&\text{ if }n\geq 2,\\
           \frac{7}{360}&\text{ if } n=1.
    \end{cases} 
\end{equation}
 Similarly, for $0<\rho_n\leq (\frac{n+1}{2})^n$ we have that
\begin{align*}
    I_0(\rho_n)&= \int_0^{\frac{n}{2}}E\cdot ({\rho_n}^{\frac{1}{2}}-0)^2dE\\&=\frac{n^2}{8}\rho_n,
\end{align*}
so 
\begin{equation*}
    \frac{I_0(\rho_n)}{\rho_n^{\frac{n+2}{n}}}=\frac{n^2}{8}\rho_n^{-\frac{2}{n}}\geq \frac{n^2}{2(n+1)^2},
\end{equation*}
for $0<\rho_n\leq  (\frac{n+1}{2})^n$. One can verify that the lower bound obtained in \eqref{lower_bound_norm} is smaller than $\frac{n^2}{2(n+1)^2}$, for every $n\in\N$.
Therefore, 
\begin{equation*}
    \frac{I_0(\rho_n)}{\rho_n^{\frac{n+2}{n}}}\geq c'_n\vcentcolon=\begin{cases}
           \frac{n^4-6n^2+6n+8}{2n^4+16n^3+42n^2+44n+16},&\text{ if }n\geq 2,\\
           \frac{7}{360}&\text{ if } n=1.
    \end{cases},
\end{equation*}
for all $\rho_n>0$. Hence,
\begin{align}\label{eq_I(rho)_relations_2norm}
    \frac{I(\rho)}{\rho^{\frac{n+2}{n}}}=\frac{I_0(\rho_n)\cdot K_n'}{\rho_n^{\frac{n+2}{n}}\cdot (K_n')^{\frac{n+2}{n}}}\geq c'_n (K_n')^{-\frac{2}{n}},
\end{align}
for every $\rho>0$. From \eqref{eqgradientfinal_2_2} and \eqref{Irho_22} we conclude that
\begin{equation*}
    \sum_{j=1}^N\langle\slashed D\psi_j,\slashed D\psi_j\rangle\geq c'_nn!^{\frac{2}{n}}2^{1-\frac{2}{n}(\lfloor\frac{n}{2}\rfloor+1)}\int_{\sphn}\rho^{\frac{n+2}{n}}(x)dx,
\end{equation*}
completing the proof.

\end{proof}

\bibliographystyle{plain}
\bibliography{references}

\end{document}